\documentclass[10pt]{amsart}
\usepackage{geometry}                
\usepackage{graphicx}
\usepackage{amssymb}
\usepackage{epstopdf}
\usepackage[colorlinks]{hyperref}
\usepackage[lite]{amsrefs}
\usepackage{mathpazo}
\usepackage{amsmath}
\usepackage{mathrsfs}
\usepackage{extarrows}
\newtheorem{theorem}{Theorem}[section]
\newtheorem{corollary}[theorem]{Corollary}
\newtheorem{lemma}[theorem]{Lemma}
\newtheorem{proposition}[theorem]{Proposition}
\theoremstyle{definition}
\newtheorem{definition}[theorem]{Definition}
\newtheorem{remark}[theorem]{Remark}
\newtheorem{example}[theorem]{Example}
\numberwithin{equation}{section}

\title{Chord index for knots in thickened surfaces}
\author{Zhiyun Cheng}
\author{Hongzhu Gao}
\author{Mengjian Xu}
\address{Laboratory of Mathematics and Complex Systems (Ministry of Education), School of Mathematical Sciences, Beijing Normal University, Beijing 100875, China}
\address{Laboratory of Mathematics and Complex Systems (Ministry of Education), School of Mathematical Sciences, Beijing Normal University, Beijing 100875, China}
\address{College of Mathematics and Statistics, Guangxi Normal University, Guilin, 541004, P. R. China}
\email{czy@bnu.edu.cn}
\email{hzgao@bnu.edu.cn}
\email{xmj@gxnu.edu.cn}
\begin{document}
\begin{abstract}
In this note, we construct a chord index homomorphism from a subgroup of $H_1(\Sigma, \mathbb{Z})$ to the group of chord indices of a knot $K$ in $\Sigma\times I$. Some knot invariants derived from this homomorphism are discussed.
\end{abstract}
\maketitle
\section{Introduction}\label{section1}
For a closed oriented surface $\Sigma$, a knot $K$ in the thickened surface is a smooth embedding of a circle $S^1$ into $\Sigma\times I$, here $I=[0, 1]$. We always assume $K$ is oriented. With respect to the projection from $\Sigma\times I$ to $\Sigma$, each knot has an associated knot diagram, i.e. a connected immersed curve in $\Sigma$ with finitely many overcrossing points and undercrossing points. Two knot diagrams represent the same knot if they are related by a finite sequence of Reidemeister moves. The investigation of knots in thickened surfaces is an important part of knot theory in general 3-manifolds. For knots in thickened surfaces, a systematic study of knot invariants derived from Gauss diagrams was carried out in \cite{Fie2001}. Following this, Grishanov and Vassiliev constructed infinitely many independent degree two invariants for knots in $\Sigma\times I$ \cite{GV2008}, and later in \cite{GV2009} they provided combinatorial formulas for these finite type invariants. The relation of cobordism for knots in $\Sigma\times I$ was introduced by Turaev in \cite{Tur2008}, which has a strong influence on the subsequent study of virtual knot concordance.

Following \cite{BR2021}, the parity axioms for knot diagrams on surfaces can be stated as follows. Consider a knot diagram $D$ on $\Sigma$, we use $C(D)$ to denote the set of crossing points of $D$. For each crossing point $c\in C(D)$, a \emph{parity} is a function $p: C(D)\to \mathbb{Z}_2(=\mathbb{Z}/2\mathbb{Z})$ which assigns 0 or 1 to $c$ such that all the following axioms are satisfied:
\begin{enumerate}
\item If a crossing point $c$ is involved in the first Reidemeister move, then $p(c)=0$;
\item If two crossing points $c_1$ and $c_2$ are involved in the second Reidemeister move, then $c_1$ and $c_2$ have the same parity;
\item If three crossing points $c_1, c_2$ and $c_3$ are involved in the third Reidemeister move, then the parity of each crossing is preserved under the move and $p(c_1)+p(c_2)+p(c_3)=0$ (mod 2).
\item For any crossing point that is not involved in a Reidemeister move, the parity of it is preserved under this move.
\end{enumerate}

For a given parity $p$, we say a crossing $c$ is even if $p(c)=0$, otherwise we say $c$ is odd. We remark that the axioms of parity were first proposed by Manturov in \cite{Man2010} for virtual knots and some other knotted objects, based on the idea of odd writhe invariant introduced by Kauffman in \cite{Kau2004}. As an extension of the parity axioms, we consider the following chord index axioms. By a ($\mathbb{Z}$-valued) \emph{chord index}, we mean a function $f: C(D)\to \mathbb{Z}$ which satisfies the following axioms:
\begin{enumerate}
\item If a crossing point $c$ is involved in the first Reidemeister move, then $f(c)=0$;
\item If two crossing points $c_1$ and $c_2$ are involved in the second Reidemeister move, then $c_1$ and $c_2$ have the same index;
\item If three crossing points $c_1, c_2$ and $c_3$ are involved in the third Reidemeister move, then the indices of them are preserved respectively under the move;
\item For any crossing point that is not involved in a Reidemeister move, the index of it is preserved under this move.
\end{enumerate}

The chord axioms above guarantee that the chord index of each crossing point of a knot diagram behaves quite well under Reidemeister moves. Therefore one can easily obtain a knot invariant for knots in $\Sigma\times I$ with a given chord index.

Consider a fixed oriented knot diagram $D$ on $\Sigma$, let us use $\mathcal{C}_{\Sigma}(D)$ to denote the set of all chord indices of $D$. Obviously, $\mathcal{C}_{\Sigma}(D)$ is an abelian group under the addition $(f_1+f_2)(c)=f_1(c)+f_2(c)$ $(\forall c\in C(D))$ and the trivial chord index $f(c)=0$ $(\forall c\in C(D))$ is the identity. We set $\mathcal{H}_1^D(\Sigma, \mathbb{Z})=\{\alpha\in H_1(\Sigma, \mathbb{Z})|\alpha\cdot[D]=0\}$, which is a subgroup of $H_1(\Sigma, \mathbb{Z})$. Here $[D]$ is the associated element of $D$ in $H_1(\Sigma, \mathbb{Z})$ and $\alpha\cdot[D]$ denotes the algebraic intersection number of the homology classes $\alpha$ and $[D]$. The main result of this note is as follows.

\begin{theorem}\label{theorem1.1}
Let $D$ be an oriented knot diagram on a closed oriented surface $\Sigma$, there exists a homomorphism $h: \mathcal{H}_1^D(\Sigma, \mathbb{Z})\to \mathcal{C}_{\Sigma}(D)$.
\end{theorem}

The article is organized as follows. In Section \ref{section2} we explain how to define a chord index $f_{\gamma}$ for a given closed curve $\gamma$ satisfying $[\gamma]\in\mathcal{H}_1^D(\Sigma, \mathbb{Z})$. Then we show that $f_{\gamma}$ only depends on the homology class $[\gamma]$. Based on this chord index, we define a knot invariant $W_K^{\gamma}(t)$ for a given knot $K$ in $\Sigma\times I$. Section \ref{section3} is devoted to study some basic properties of the chord index $f_{\gamma}$ and the knot invariant $W_K^{\gamma}(t)$. Section \ref{section4} contains a generalized version of the chord index $f_{\gamma}$, which take values in the group ring $\mathbb{Z}H_1(\Sigma, \mathbb{Z})$. Knots in thickened surfaces are closely connected with virtual knots, however the invariant $W_K^{\gamma}(t)$ defined in Section \ref{section2} cannot be defined for virtual knots directly. Some applications of the chord index $f_{\gamma}$ in virtual knot theory are given in Section \ref{section5}.

\section{Chord index derived from $\mathcal{H}_1^D(\Sigma, \mathbb{Z})$}\label{section2}
The main idea of defining a chord index with respect to a closed curve is inspired by the $\mathscr{C}$-parity introduced by Boden and Rushworth in \cite{BR2021}. Consider an oriented knot diagram $D$ on a closed oriented surface $\Sigma$, choose a closed oriented curve $\gamma$ on $\Sigma$ such that $[\gamma]\in\mathcal{H}_1^D(\Sigma, \mathbb{Z})$. Perturb $D$ or $\gamma$ so that $D$ meets $\gamma$ transversally in $n$ points. These $n$ points divide the knot diagram $D$ into $n$ parts. At a point $x\in D\cap\gamma$, without loss of generalization, we say $x$ is positive if the pair of tangent vectors of $\gamma$ and $D$ at $x$ orients $T_x\Sigma$. Otherwise, we say $x$ is negative. The coloring rule is when we walk along $D$ according to the orientation, the integer increases(decreases) by one if we meet a positive(negative) intersection point. The local condition of the coloring is illustrated in Figure \ref{figure1}.
\begin{figure}[h]
\centering
\includegraphics{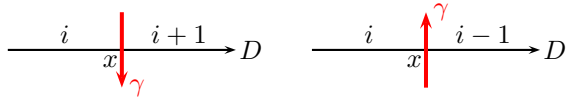}\\
\caption{A coloring of $D$ with respect to $\gamma$}\label{figure1}
\end{figure}

\begin{lemma}\label{lemma2.1}
The coloring above is well-defined modulo the addition of a constant integer.
\end{lemma}
\begin{proof}
One can freely choose any part of $D$ and assigns an integer to it. The assumption $[\gamma]\cdot[D]=0$ guarantees that when we return to the beginning point, the new coloring coincides with the original one.
\end{proof}

\begin{definition}
Choose a closed curve $\gamma\subset\Sigma$ which satisfies $[\gamma]\in\mathcal{H}_1^D(\Sigma, \mathbb{Z})$, we define a map $f_{\gamma}: C(D)\to\mathbb{Z}$ which assigns the difference between the integer on the over-arc and the integer on the under-arc to each crossing point.
\end{definition}

Note that although the coloring induced by $\gamma$ is not well-defined, Lemma \ref{lemma2.1} tells us that the assignment $f_{\gamma}(c)$ is well-defined. We will see soon that $f_{\gamma}$ is a chord index.

\begin{lemma}\label{lemma2.2}
The assignment $f_{\gamma}$ satisfies all the chord index axioms, hence it is a chord index.
\end{lemma}
\begin{proof}
It suffices to check each Reidemeister moves respectively. Notice that each Reidemeister move happens in a small disk of $\Sigma$. If $\gamma$ does not appear in this disk, then according to the coloring rule, the color of each arc in this disk is preserved respectively. It is easy to observe that $f_{\gamma}$ is a chord index.

If $\gamma$ does appear in the disk, the move involving $\gamma$ could be quite complicated. However, if we consider $D$ and $\gamma$ as an entirety, this move can be decomposed into a finite number of Reidemeister moves involving both $D$ and $\gamma$. Therefore it is sufficient to check all the cases of this kind of Reidemeister moves.

Since the first Reidemeister move concerns only one component, we only need the check the second and the third Reidemeister moves. In Figure \ref{figure2}, we list all the cases that both $D$ and $\gamma$ are both involved. Similar to Figure \ref{figure1}, here the curves $\gamma$ are depicted in red. In each case, we find that the integers associated to the endpoints on the boundary of the small disk are invariant and the chord indices of all crossing points of $D$ are preserved. For example, the two crossing points on the bottom left of Figure \ref{figure2} both have index $(j+1)-(i+1)=j-i$. Note that the orientations of some arcs in Figure \ref{figure2} can be reversed and the two crossing points can be switched, the verification for any of these cases is analogous.
\begin{figure}[h]
\centering
\includegraphics{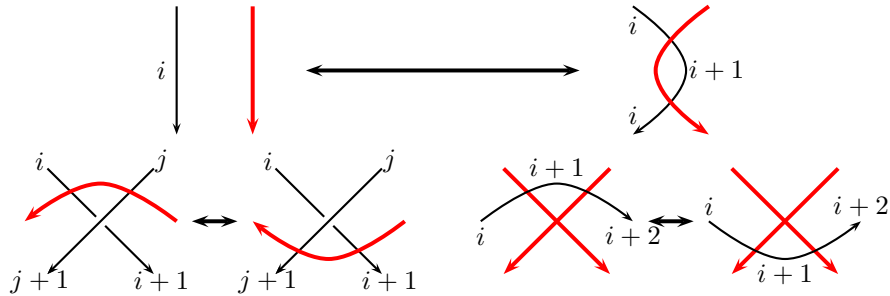}\\
\caption{Reidemeister moves involving both $D$ and $\gamma$}\label{figure2}
\end{figure}
\end{proof}

Before proving Theorem \ref{theorem1.1}, we need another interpretation of the chord index $f_{\gamma}(c)$. This definition can be regarded as an extension of the topological definition given in \cite{FK2013} and \cite{BCG2019}. Consider a crossing point $c\in C(D)$, by smoothing $c$ according to the orientation we obtain two knot diagrams, say $D_c^l$ and $D_c^r$, see Figure \ref{figure3}.
\begin{figure}[h]
\centering
\includegraphics{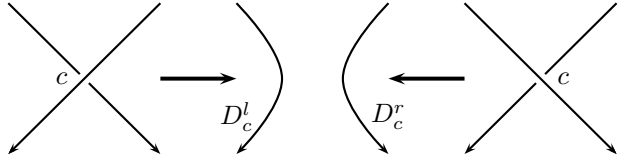}\\
\caption{Smoothing a crossing point $c$}\label{figure3}
\end{figure}

\begin{lemma}\label{lemma2.3}
Let $c$ be a crossing point of an oriented knot diagram $D$ on $\Sigma$, $\gamma$ a closed curve that satisfies $[\gamma]\in\mathcal{H}_1^D(\Sigma, \mathbb{Z})$, then
$$f_{\gamma}(c)=
\begin{cases}
[\gamma]\cdot[D_c^r],& \text{if } w(c)=+1;\\
[\gamma]\cdot[D_c^l],& \text{if } w(c)=-1.
\end{cases}$$
Here $w(c)$ denotes the writhe of the crossing point $c$.
\end{lemma}
\begin{proof}
Let us assume $w(c)=+1$, the negative case can be proved similarly. Since $f_{\gamma}(c)$ is defined as the integer on the over-arc minus the integer on the under-arc, notice that each positive(negative) intersection point between $\gamma$ and $D_c^r$ increases(decreases) this difference by one, the result follows directly.
\end{proof}

\begin{remark}\label{remark2.4}
We would like to remark that the chord index $f_{\gamma}(c)$ can be rewritten as
\begin{center}
$f_{\gamma}(c)=w(c)([\gamma]\cdot[D_c^r])$,
\end{center}
since $[\gamma]\cdot([D_c^l]+[D_c^r])=\gamma\cdot[D]=0$.
\end{remark}

Now we turn to the proof of Theorem \ref{theorem1.1}.
\begin{proof}
For arbitrary element $\alpha\in\mathcal{H}_1^D(\Sigma, \mathbb{Z})$, choose an oriented closed curve $\gamma$ such that $[\gamma]=\alpha$. In general, $\gamma$ is not a simple closed curve, since a nonzero homology class can be represented by a simple closed curve if and only if it is primitive \cite{Mey1976}.

Now for each crossing point $c\in C(D)$, we define the \emph{chord index homomorphism} as $h(\alpha)(c)=f_{\gamma}(c)$. We claim this definition is well-defined. Actually, if $\gamma'$ is another closed curve representing $\alpha$, then
\begin{center}
$f_{\gamma}(c)-f_{\gamma'}(c)=w(c)([\gamma]\cdot[D_c^r])-w(c)([\gamma']\cdot[D_c^r])=w(c)(([\gamma]-[\gamma'])\cdot[D_c^r])=0$.
\end{center}

In order to see that $h$ is a homomorphism, we choose another element $\beta\in\mathcal{H}_1^D(\Sigma, \mathbb{Z})$ and an oriented closed curve $\delta$ representing $\beta$. It follows that
\begin{center}
$h(\alpha+\beta)(c)=f_{\gamma\cup\delta}(c)=w(c)([\gamma\cup\delta]\cdot[D_c^r])=w(c)([\gamma]\cdot[D_c^r]+[\delta]\cdot[D_c^r])=h(\alpha)(c)+h(\beta)(c)$.
\end{center}
The proof is finished.
\end{proof}

\begin{remark}\label{remark2.5}
As we mentioned in the proof above, in general the curve $\gamma$ cannot be chosen as a simple closed curve. However, by oriented smoothing all the self-intersection points of $\gamma$, one obtains a finite number of disjoint simple closed curves $\cup\gamma_i$. Since $[\gamma]=\sum\limits_i[\gamma_i]$, if we use $\cup\gamma_i$ to define a chord index, then we have $f_{\cup\gamma_i}=f_{\gamma}$. Therefore, it is not necessary to realize $\alpha\in\mathcal{H}_1^D(\Sigma, \mathbb{Z})$ by a connected closed curve.
\end{remark}

For an oriented knot $K$ in $\Sigma\times I$, the homology class $[K]$ does not depend on the choice of the knot diagram $D$. Hence it is safe to write $\mathcal{H}_1^K(\Sigma, \mathbb{Z})$ instead of $\mathcal{H}_1^D(\Sigma, \mathbb{Z})$. With a given element $\alpha\in\mathcal{H}_1^K(\Sigma, \mathbb{Z})$, by choosing an oriented closed curve $\gamma$ that represents $\alpha$, one obtains a chord index $f_{\gamma}$. Now we can associate a polynomial $W_K^{\gamma}(t)\in\mathbb{Z}[t, t^{-1}]$ to the knot $K$, which is defined as
\begin{center}
$W_K^{\gamma}(t)=\sum\limits_{f_{\gamma}(c)\neq0}w(c)t^{f_{\gamma}(t)}$.
\end{center}
We call this polynomial the \emph{writhe polynomial} of $K$ with respect to $\gamma$, or $[\gamma]$ more precisely, adopting the terminology used in \cite{CG2013}. The relation between $W_K^{\gamma}(t)$ and the writhe polynomial of virtual knots will be discussed in Section \ref{section5}.

We end this section with a simple example.

\begin{example}\label{example2.6}
Consider the knot $K\subset T^2\times I$ and the closed curve $\gamma$ depicted in red, see Figure \ref{figure4}. Let us use $c_1$ and $c_2$ to denote the two crossing points of $K$. Direct calculation shows that $f_{\gamma}(c_1)=-1$ and $f_{\gamma}(c_2)=1$. It follows that the writhe polynomial of $K$ with respect to $\gamma$ has the form $W_K^{\gamma}(t)=-t-t^{-1}$. Recall that instead of one closed curve, we can also use a finite family of closed curve to represent an element in $\mathcal{H}_1^K(\Sigma, \mathbb{Z})$, see Remark \ref{remark2.5}. If we replace $\gamma$ with $n$ simple closed curves that are all parallel to $\gamma$, then the writhe polynomial $W_K^{n\gamma}(t)=-t^n-t^{-n}$.
\begin{figure}[h]
\centering
\includegraphics{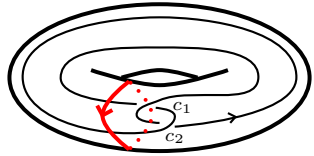}\\
\caption{A null-homologous knot in $T^2\times I$}\label{figure4}
\end{figure}
\end{example}

\section{Some properties of $f_{\gamma}$ and $W_K^{\gamma}(t)$}\label{section3}
In this section we study some basic properties of the chord index $f_{\gamma}$ and the associated writhe polynomial $W_K^{\gamma}(t)$. For the sake of simplicity, we will often abuse our notation in that we denote by $K$ both a knot diagram and the knot it represents.

\begin{proposition}\label{proposition3.1}
Let $K$ be an oriented knot in $\Sigma\times I$ and $\gamma$ an oriented closed curve on $\Sigma$ which satisfies $[\gamma]\in\mathcal{H}_1^K(\Sigma, \mathbb{Z})$. Then we have $f_{-\gamma}=-f_{\gamma}$ and $W_K^{-\gamma}(t)=W_K^{\gamma}(t^{-1})$. In particular, if $[\gamma]=0$, then $f_{\gamma}$ is the trivial chord index.
\end{proposition}
\begin{proof}
The first equality follows immediately from the fact that $h$ is a homomorphism. Or more directly, one observes that replacing each integer on the knot diagram with its opposite provides a coloring of the same diagram with respect to $-\gamma$. The second equality follows directly from the definition.
\end{proof}

Next we turn to use $W_K^{\gamma}(t)$ to investigate the symmetry properties of the knot $K$.

\begin{proposition}\label{proposition3.2}
Let $K$ be an oriented knot in $\Sigma\times I$ and $r(K)$ be the knot obtained from $K$ by reversing the orientation, then for any $[\gamma]\in\mathcal{H}_1^K(\Sigma, \mathbb{Z})$, we have $W_{r(K)}^{\gamma}(t)=W_K^{\gamma}(t)$.
\end{proposition}
\begin{proof}
It suffices to notice that for any fixed $\gamma$, the coloring for $K$ also gives a coloring for $r(K)$. Together with the fact that the same crossing point in $K$ and $r(K)$ have the same writhe, the result follows immediately.
\end{proof}

Proposition \ref{proposition3.2} implies that we can not use $W_K^{\gamma}(t)$ to detect whether a knot in $\Sigma\times I$ is invertible directly. Nevertheless, if we make the choice of the curve $\gamma$ more flexible, it is still possible to tell the difference between $K$ and $r(K)$ by using the writhe polynomial.

\begin{corollary}\label{corollary3.3}
Let $K$ and $r(K)$ be as in Proposition \ref{proposition3.2}, for any closed curves $\gamma_K$ and $\gamma_{r(K)}$ which satisfy $[\gamma_K]=[K]$ and $[\gamma_{r(K)}]=[r(K)]$, we have $W_{r(K)}^{\gamma_{r(K)}}(t)=W_K^{\gamma_K}(t^{-1})$.
\end{corollary}
\begin{proof}
Recall that the intersection form $H_1(\Sigma,\mathbb{Z})\times H_1(\Sigma,\mathbb{Z})\to\mathbb{Z}$ is anti-symmetric, therefore $[\gamma_K]\cdot[K]=[K]\cdot[K]=0$ and $[\gamma_{r(K)}]\cdot[r(K)]=[r(K)]\cdot[r(K)]=0$. One calculates
\begin{center}
$W_{r(K)}^{\gamma_{r(K)}}(t)\xlongequal{\text{Proposition \ref{proposition3.2}}}W_K^{\gamma_{r(K)}}(t)=W_K^{-\gamma_K}(t)\xlongequal{\text{Proposition \ref{proposition3.1}}}W_K^{\gamma_K}(t^{-1})$.
\end{center}
\end{proof}

\begin{proposition}\label{proposition3.4}
Let $K$ be an oriented knot in $\Sigma\times I$ and $m(K)$ be the knot obtained from $K$ by switching all the crossing points, then for any $[\gamma]\in\mathcal{H}_1^K(\Sigma, \mathbb{Z})$, we have $W_{m(K)}^{\gamma}(t)=-W_K^{\gamma}(t^{-1})$.
\end{proposition}
\begin{proof}
According to the coloring rule defined in Section \ref{section2}, for a fixed closed curve $\gamma$ the integers on $K$ and those on $m(K)$ coincide with each other. However, switching a crossing point not only interchanges the over-arc and under-arc, but also changes the writhe of this crossing. It follows that $W_{m(K)}^{\gamma}(t)=-W_K^{\gamma}(t^{-1})$.
\end{proof}

\begin{example}
Let us consider the knot diagram $K$ on $S^2$ with three 1-handles $H_i=S^1\times I$ $(1\leq i\leq3)$ attached along $\partial H_i=\partial H_i^+\cup\partial H_i^-$, see Figure \ref{figure11}. Choose a closed $\gamma$ which satisfies $[\gamma]=[K]$, direct calculation shows that $f_{\gamma}(c_1)=f_{\gamma}(c_2)=1$ and $f_{\gamma}(c_3)=-2$. As a corollary, we have $W_K^{\gamma}(t)=2t+t^{-2}, W_{r(K)} ^{-\gamma}(t)=t^2+2t^{-1}$ and $W_{m(K)}^{\gamma}(t)=-t^2-2t^{-1}$. According to Corollary \ref{corollary3.3} and Proposition \ref{proposition3.4}, we conclude that $K\neq r(K)$ and $K\neq m(K)$.
\begin{figure}[h]
\centering
\includegraphics{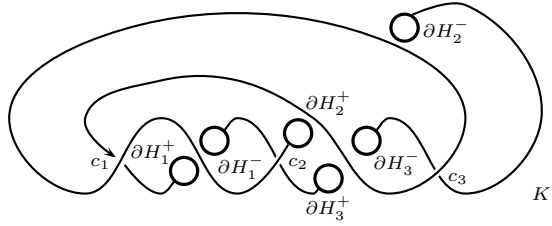}\\
\caption{A non-invertible and chiral knot $K$ in a thickened surface}\label{figure11}
\end{figure}
\end{example}

Let $K_1$ and $K_2$ be two oriented knots in $\Sigma\times I$ and $b: [0, 1]\times[0, 1]\to\Sigma\times I$ an embedding such that $b([0, 1]\times[0, 1])\cap K_1=b([0, 1]\times\{0\})$ and $b([0, 1]\times[0, 1])\cap K_2=b([0, 1]\times\{1\})$. By removing $b([0, 1]\times\{0,1\})$ and adding $b(\{0, 1\}\times[0, 1])$ to $K_1\cup K_2$ one obtains a new oriented knot, which is called the \emph{band sum} of $K_1$ and $K_2$ and denoted by $K_1\#_bK_2$.

\begin{proposition}\label{proposition3.5}
Let $K_1\#_bK_2$ be a band sum of $K_1$ and $K_2$ with a band $b$ in $\Sigma\times I$. Assume there exists a closed curve $\gamma$ on $\Sigma$ which satisfies $[\gamma]\cdot[K_1]=[\gamma]\cdot[K_2]=0$, and there exists no crossing point between $K_1$ and $K_2$, then $W_{K_1\#_bK_2}^{\gamma}(t)=W_{K_1}^{\gamma}(t)+W_{K_2}^{\gamma}(t)$.
\end{proposition}
\begin{proof}
First notice that $[\gamma]\cdot[K_1\#_bK_2]=[\gamma]\cdot[K_1]+[\gamma]\cdot[K_2]=0+0=0$, hence the writhe polynomial $W_{K_1\#_bK_2}^{\gamma}(t)$ makes sense.

Consider a crossing point $c\in C(K_1)$, there exists a corresponding crossing point in $C(K_1\#_bK_2)$ and we still use $c$ to denote it. Let $f_{\gamma}^{K_1}(c)$ and $f_{\gamma}^{K_1\#_bK_2}(c)$ denote the chord index of $c$ in $C(K)$ and $C(K_1\#_bK_2)$ respectively. Then we have $f_{\gamma}^{K_1\#_bK_2}(c)=w(c)([\gamma]\cdot[(K_1\#_bK_2)_c^r])=w(c)([\gamma]\cdot([(K_1)_c^r]+\epsilon[K_2]))=w(c)([\gamma]\cdot[(K_1)_c^r])=f_{\gamma}^{K_1}(c)$, where $\epsilon\in\{0, 1\}$. More precisely, $\epsilon=1$ if $b([0, 1]\times\{0\})\subset (K_1)_c^r$, otherwise $\epsilon=0$. Similarly, one can prove that $f_{\gamma}^{K_1\#_bK_2}(c')=f_{\gamma}^{K_2}(c')$ if $c'\in C(K_2)$.

In order to complete the proof, it suffices to show that the contribution coming from the ``new" crossing points formed by $b(\{0, 1\}\times[0, 1])$ with itself or $b(\{0, 1\}\times[0, 1])$ with $K_i$ ($i=1, 2$) vanishes. In order to see this, we notice that these new crossing points always come in pairs locally, either one is formed by an arc with $b(\{0\}\times[0, 1])$, the other one is formed by the same arc with $b(\{1\}\times[0, 1])$, or (after an isotopy if necessary) two crossing points between $b(\{0\}\times[0, 1])$ and $b(\{1\}\times[0, 1])$, just like the two crossings in a full twist. For the first case, similar to the proof of Proposition \ref{proposition3.2}, it is easy to observe that locally these two parallel arcs coming from $b(\{0\}\times[0, 1])$ and $b(\{1\}\times[0, 1])$ have the same color with respect to $\gamma$. Hence these two crossing points have the same chord index but different writhes, and their contributions to the writhe polynomial cancel out. For the second case, notice that the two arcs in the full twist have the same color, therefore these two crossing points both have chord index zero. It follows that all these new crossing points together has no contribution to $W_{K_1\#_bK_2}^{\gamma}(t)$.
\end{proof}

We remark that we can not drop the assumption that the two knot diagrams $K_1, K_2$ satisfy $K_1\cap K_2=\emptyset$, just like when we consider the band sum of two knots in $S^3$ we usually require that these two knots form a 2-component split link. For example, the knot $K$ described in Example \ref{example2.6} can be regarded as the band sum of the two components of a Hopf link. Although each component is a trivial knot and hence has writhe polynomial zero, the writhe polynomial of $K$ is nontrivial.

\section{Generalized chord index valued in $\mathbb{Z}H_1(\Sigma, \mathbb{Z})$}\label{section4}
This section includes some remarks on the $\mathbb{Z}$-valued chord index defined in Section \ref{section1}. First, if we consider homology groups with coefficients in $\mathbb{Z}_2$ rather than $\mathbb{Z}$, then all the theory we developed above carries out directly with no change in the proofs. More precisely, for a knot $K$ in $\Sigma\times I$ and a closed curve $\gamma$ that satisfies $[K]\cdot[\gamma]=0$ (mod 2), then the arcs of $K$ can be colored with $\{0, 1\}$ such that these two colors appear alternatively. We say a crossing point $c\in C(K)$ is \emph{odd} if the two arcs of this crossing has different colors, otherwise we say it is \emph{even}. This is exactly the parity introduced by Boden and Rushworth in \cite{BR2021} and it satisfies all the parity axioms mentioned in Section \ref{section1}.

Strictly speaking, our definition of the parity axioms is a bit different from that introduced by Boden and Rushworth. For the third Reidemeister move, we require the parity $p$ satisfies $\sum\limits_{i=1}^3p(c_i)=0$ (mod 2). However, in Boden and Rushworth's definition, the case $p(c_1)=p(c_2)=p(c_3)=1$ (mod 2) is also allowed. In this case, one can also apply the \emph{parity projection} operation, which is a well defined operation sending a virtual knot to another one. The reason why we delete it from our definition is this case never happens in the homological parity construction.

In order to see this, we suffices ourselves with the following two example verifications, other cases can be verified in a similar manner. For the diagram on the left side of Figure \ref{figure5}, evident calculation shows that $[D_{c_1}^r]+[D_{c_2}^r]+[D_{c_3}^r]=[D]$. For the diagram on the right side, one observes that $[D_{c_1}^r]+[D]=[D_{c_2}^r]+[D_{c_3}^r]$, see also \cite[Lemma 1.3]{Fie2001}. In both cases we have $f_{\gamma}(c_3)=f_{\gamma}(c_1)+f_{\gamma}(c_2)$. In particular, any chord parity derived from $\mathcal{H}_1^D(\Sigma, \mathbb{Z}_2)$ can not provide a parity such that $c_1, c_2$ and $c_3$ are all odd.
\begin{figure}[h]
\centering
\includegraphics{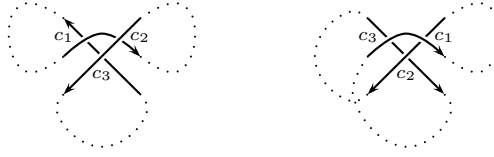}\\
\caption{Two cases of the third Reidemeister move}\label{figure5}
\end{figure}

Now we turn to discuss an extension of the $\mathbb{Z}$-valued chord index defined in Section \ref{section2}. The following definition of chord index is an analogue of the generalized chord index proposed by the first author in \cite{Che2021}.

As before, let $D$ be an oriented knot diagram on a closed oriented surface $\Sigma$, and $C(D)$ be the set of crossing points of $D$. A (generalized) chord index is a function $g: C(D)\to G$ where $G$ is a set, such that all the following axioms are satisfied:
\begin{enumerate}
\item Any crossing point involved in the first Reidemeister move has a fixed chord index;
\item The two crossing points involved in the second Reidemeister move have the same chord index;
\item The chord indices of the three crossing points involved in the third Reidemeister move are preserved respectively under this move;
\item The chord index of any crossing point that is not involved in a Reidemeister move is preserved under this move.
\end{enumerate}
Under this generalized definition, the chord indices of crossing points are no longer restricted to integers. The following is an example, which can be considered as a generalization of the writhe polynomial $W_K^{\gamma}(t)$ defined in Section \ref{section2}. The main thoughts of it is similar to the small state sum invariant defined by Thomas Fiedler in \cite{Fie1993}, see also \cite{CGX2020}.

Consider a crossing point $c\in C(D)$, we define $g(c)=[D_c^l]+[D_c^r]\in G=\mathbb{Z}[H_1(\Sigma, \mathbb{Z})]$, the group ring of $H_1(\Sigma, \mathbb{Z})$ over $\mathbb{Z}$. It is a routine exercise to verify that $g$ satisfies all the axioms above. In particular, for a crossing point involved in the first Reidemeister move, the chord index of it equals $[D]+[0]=[D]$. Since the homology class of $K$ is independent of the choice of the knot diagram, one obtains an invariant from the chord index $g$, which can be defined as
\begin{center}
$\mathcal{W}_K=\sum\limits_{c\in C(D)}w(c)g(c)-w(D)[K]\in\mathbb{Z}[H_1(\Sigma, \mathbb{Z})]$,
\end{center}
here $w(D)$ denotes the writhe of the knot diagram $D$. The chord index axioms above guarantee that $\mathcal{W}_K$ is independent of the choice of the knot diagram.

\begin{remark}\label{remark4.1}
In \cite{Fie1993}, Fiedler introduced the small state sum invariant for knots in thickened surfaces. In our language, Fiedler considered the function
$$g_F(c)=
\begin{cases}
[D_c^l],& \text{if } w(c)=+1;\\
[D_c^r],& \text{if } w(c)=-1.
\end{cases}$$
This function satisfies all the axioms above but not the first one. Actually, for a crossing point $c$ in the first Reidemeister move, $g_F(c)=[K]$ or $[0]$. In order to overcome this obstacle, Fiedler identified the elements $[K]$ and $[0]$ in $H_1(\Sigma, \mathbb{Z})$. In other words, by choosing $G=\mathbb{Z}[H_1(\Sigma, \mathbb{Z})/<[K]>]$, the small state sum invariant can be defined as $\sum\limits_{c\in C(D)}w(c)g_F(c)-w(D)[0]$. We remark that Fiedler also considered the case that $\Sigma$ is non-orientable.
\end{remark}

\begin{remark}\label{remark4.2}
Recall that two knot diagrams are \emph{regular isotopy} if they are related by some Reidemeister moves of type II and type III. If we want to define a regular invariant derived from chord index, the first axiom can be dropped. As an example, for each crossing point $c\in C(D)$ we can assign an index
$$g_{reg}(c)=
\begin{cases}
x[D_c^l]+y[D_c^r],& \text{if } w(c)=+1;\\
y[D_c^l]+x[D_c^r],& \text{if } w(c)=-1
\end{cases}$$
to it. It is not difficult to verify that $g_{reg}(c)$ satisfies all the axioms above but the first one. As a consequence, the sum $\mathcal{W}_K^{reg}=\sum\limits_{c\in C(D)}w(c)g_{reg}(c)\in\mathbb{Z}[H_1(\Sigma, \mathbb{Z}[x, y])]$ is a regular invariant of knots in $\Sigma\times I$.
\end{remark}

\begin{remark}\label{remark4.3}
All the discussion above, including the writhe polynomial $W_K^{\gamma}(t)$ and its generalization $\mathcal{W}_K$, can be extended directly to the case that $\Sigma$ has nonempty boundaries.
\end{remark}

\section{Applications for virtual knots}\label{section5}
We begin this section with a quick review of virtual knots. As an extension of the classical knot theory, virtual knot theory was introduced by Kauffman in \cite{Kau1999}. We know that a classical knot is an embedding of $S^1$ in $R^3$ up to isotopy. It is equivalent to consider embeddings of $S^1$ in $S^2\times I$, which yields the same theory. By replacing the 2-sphere with surfaces of higher genera one obtains the virtual knot theory. More precisely, a virtual knot is an embedding of $S^1$ in $\Sigma\times I$. We say two virtual knots are \emph{equivalent} if they are related by an isotopy in $\Sigma\times I$, homeomorphisms of surfaces and addition or subtraction an empty handle. From the viewpoint of knot diagram, a virtual knot diagram is a classical knot diagram with some classical crossing points replaced by virtual crossing points, and each virtual crossing point is usually denoted by a small circle around the vertex. We say two virtual knot diagrams are \emph{equivalent} if they are related by some generalized Reidemeister moves, see Figure \ref{figure6}. Each virtual knot diagram corresponds to a knot in a thickened surface. In order to see this, one just needs to image the virtual knot diagram as a diagram on $S^2$, then adds a handle for each virtual crossing point to eliminate this crossing on the surface. It is known that two virtual knots are equivalent if and only if they represent the same virtual knot \cite{Kau1999}.

\begin{figure}
\centering
\includegraphics{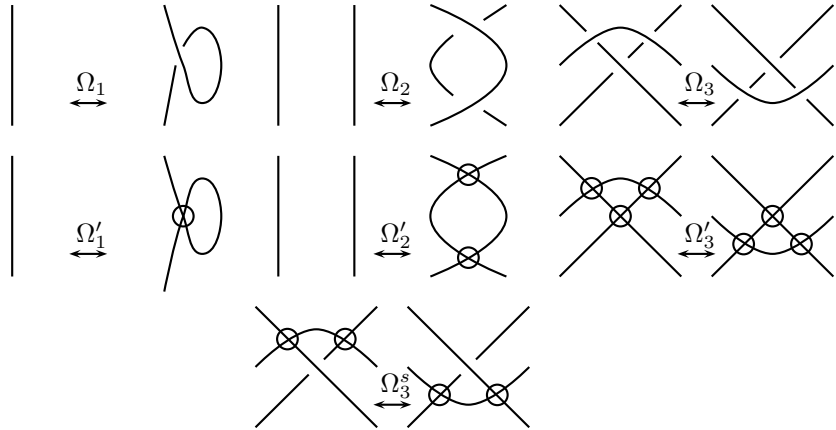}\\
\caption{Generalized Reidemeister moves}\label{figure6}
\end{figure}

Now we recall the definition of writhe polynomial, which was introduced in \cite{CG2013}, see also \cite{Dye2013,ISL2013,Kau2013,ST2014}. Given a virtual knot diagram, there exists a corresponding Gauss diagram which can be regarded as the pre-image of the knot diagram. A Gauss diagram is a counterclockwise oriented planar circle with some oriented signed chords inside. Each chord corresponds to a classical crossing point of the diagram, and it is directed from the pre-image of the overcrossing to the pre-image of the undercrossing. Finally, the sign of a chord is nothing but the writhe of the corresponding crossing point. Unlike classical knot diagrams, each Gauss diagram can be realized by infinitely many different virtual knot diagrams. Although the realization is not unique, all these virtual knot diagrams represent the same virtual knot. One benefit of considering a Gauss diagram rather than the corresponding virtual knot diagram is, there is no virtual crossing point there and therefore one only needs to consider the classical Reidemeister moves.

Assume we are given a virtual knot diagram $K$, let us use $G(K)$ to denote the associated Gauss diagram. For a classical crossing point $c$, we use the same notation to denote the corresponding chord in $G(K)$. Consider other chords in $G(K)$ which has nonempty intersection with $c$, we define
\begin{itemize}
  \item $r_+(c)$ to be the number of positive chords crossing $c$ from left to right;
  \item $r_-(c)$ to be the number of negative chords crossing $c$ from left to right;
  \item $l_+(c)$ to be the number of positive chords crossing $c$ from right to left;
  \item $l_-(c)$ to be the number of negative chords crossing $c$ from right to left.
\end{itemize}
See Figure \ref{figure7}.
\begin{figure}[h]
\centering
\includegraphics[width=3cm]{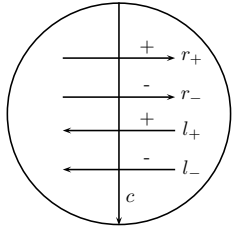}\\
\caption{The definition of the chord index}\label{figure7}
\end{figure}

Now we define the \emph{chord index} of $c$ as
\begin{center}
Ind$(c)=r_+(c)-r_-(c)-l_+(c)+l_-(c)$,
\end{center}
and the \emph{writhe polynomial} of a virtual knot $K$ can be defined as
\begin{center}
$W_K(t)=\sum\limits_{\text{Ind}(c)\neq0}w(c)t^{\text{Ind}(c)}$,
\end{center}
which is a virtual knot invariant.

\begin{remark}\label{remark5.1}
Sometimes the writhe polynomial is defined to be $\sum\limits_{c}w(c)t^{\text{Ind}(c)}-w(K)$, which equals $\sum\limits_{\text{Ind}(c)\neq0}w(c)t^{\text{Ind}(c)}+\sum\limits_{\text{Ind}(c)=0}w(c)t^{\text{Ind}(c)}-\sum\limits_{\text{Ind}(c)\neq0}w(c)-\sum\limits_{\text{Ind}(c)=0}w(c)=W_K(t)-W_K(1)$.
\end{remark}

Let $K$ be a virtual knot, which is realized as a circle in $\Sigma\times I$. The following lemma tells us that Ind$(c)$ can be recovered from $f_{\gamma}(c)$ for some specially chosen closed curve $\gamma$.

\begin{lemma}\label{lemma5.2}
For an oriented knot $K$ in $\Sigma\times I$, choose a closed curve $\gamma$ on $\Sigma$ such that $[\gamma]=[K]$, then for any crossing point $c$ we have $f_{\gamma}(c)=\text{Ind}(c)$. In particular, for this choice of $\gamma$ the two writhe polynomials $W_K^{\gamma}(t)$ and $W_K(t)$ coincide with each other.
\end{lemma}
\begin{proof}
In \cite[Lemma 3.2]{BCG2019}, it was proved that Ind$(c)=w(c)[D_c^l]\cdot[K]$. Then one calculates
\begin{center}
Ind$(c)=w(c)[D_c^l]\cdot[K]=w(c)[K]\cdot(-[D_c^l])=w(c)[K]\cdot([K]-[D_c^l])=w(c)[\gamma]\cdot[D_c^r]=f_{\gamma}(c)$.
\end{center}
The second statement follows directly.
\end{proof}

Lemma \ref{lemma5.2} tells us that $f_{\gamma}$ can be used to define a polynomial invariant for $K$ if $[\gamma]=[K]$. Nevertheless, not every closed $\gamma$ on $\Sigma$ can be used to define a virtual knot invariant. The main reason is, for virtual knots the surface $\Sigma$ is not fixed. In other words, we are allowed to add or remove empty handles. If one chooses a closed curve $\gamma$ on this kind of nugatory handle, the result obtained may be not a virtual knot invariant.

As an extreme example, consider a fixed classical knot diagram $D$ on $S^2$ with crossing points $C(D)=\{c_1, \cdots, c_n\}$. For arbitrarily chosen integers $a_1, \cdots, a_n$, if we are allowed to add new handles, then there exists an element $[\gamma]\in\mathcal{H}_1^D(\Sigma, \mathbb{Z})$ such that $f_{\gamma}(c_i)=a_i$ $(1\leq i\leq n)$. In order to see this, for each crossing point $c_i$ we add $a_i$ concentric circles around $c_i$, see Figure \ref{figure8}. Note that in Figure \ref{figure8}, $2a_i$ handles (actually, two are sufficient) are attached to the 2-sphere so that all the virtual crossing points actually are not real. It is evident that $f_{\gamma}(c_i)=a_i$, which is a positive integer. If $a_i<0$ or $w(c_i)=-1$, one needs to reverse all the orientations of the concentric circles.
\begin{figure}[h]
\centering
\includegraphics{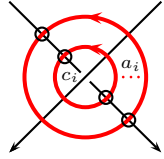}\\
\caption{A crossing point with index $a_i$}\label{figure8}
\end{figure}

If one wants to define a virtual knot invariant based on the chord index $f_{\gamma}$, one needs to fix the surface $\Sigma$. Each virtual knot $K$ can be realized as a circle in thickened surfaces, the minimal genus of these surfaces is called the \emph{supporting genus} of $K$. It was proved by Kuperberg \cite{Kup2003} that the embedding of a virtual knot in the thickened surface of minimal genus is unique. As a consequence, it is a natural choice to consider the minimal surface for a given virtual knot. The result below follows immediately.

\begin{proposition}\label{proposition5.3}
Let $K$ be a virtual knot which possesses a representative in thickened surface $\Sigma\times I$, if the genus of $\Sigma$ is equal to the supporting genus of $K$, then for any closed curve $\gamma$ satisfying $[\gamma]\in\mathcal{H}_1^K(\Sigma, \mathbb{Z})$, the writhe polynomial $W_K^{\gamma}(t)$ is a virtual knot invariant.
\end{proposition}

The following example, together with Lemma \ref{lemma5.2}, show that $W_K^{\gamma}(t)$ is strictly stronger than $W_K(t)$.

\begin{example}\label{example5.4}
Consider the Kishino knot $K$ depicted in Figure \ref{figure9}, which is a nontrivial virtual knot although it can be regarded as the connected sum of two trivial knots. Many virtual knot invariants fail to detect the nontriviality of it, including the writhe polynomial $W_K(t)$. It is known that the supporting genus of it equals two \cite[Theorem 4.1]{DK2005}, see also \cite[Example 7.3]{CSW2014}. Consider the closed $\gamma$ colored in red, which has four intersection points with $K$. Notice that the two virtual crossings correspond to two handles. Direct calculation shows that $f_{\gamma}(c_1)=f_{\gamma}(c_2)=0$, $f_{\gamma}(c_3)=-1$ and $f_{\gamma}(c_4)=1$, then we have $W_K^{\gamma}(t)=-t-t^{-1}$.
\begin{figure}[h]
\centering
\includegraphics{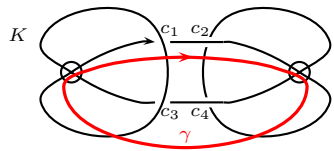}\\
\caption{Kishino knot and a closed curve $\gamma$}\label{figure9}
\end{figure}
\end{example}

\begin{corollary}\label{corollary5.5}
Let $K$ be a virtual knot represented by a knot diagram $D$ on $\Sigma$, if for any nontrivial $[\gamma]\in\mathcal{H}_1^K(\Sigma, \mathbb{Z})$ the writhe polynomial $W_K^{\gamma}(t)\neq0$, then $\Sigma$ is of minimal genus.
\end{corollary}
\begin{proof}
If not, suppose the genus of $\Sigma$ is greater than the supporting genus of $K$. Then $D$ can be isotopic to another knot diagram $D'$ such that there exists an essential non-separating closed curve $\gamma\subset\Sigma$, which has no intersection with $D'$ \cite{Kup2003}. Obviously, the homology class $[\gamma]\neq[0]$ and $[\gamma]\in\mathcal{H}_1^K(\Sigma, \mathbb{Z})$. Since $\gamma\cap D'=\emptyset$, all crossing points of $D'$ have chord index zero, it follows that $W_K^{\gamma}(t)=0$. This contradicts the hypothesis that the writhe polynomial $W_K^{\gamma}(t)\neq0$.
\end{proof}

In \cite{Che2017}, the first author introduced a transcendental function invariant for virtual knots, which can be considered as a generalization of the writhe polynomial $W_K(t)$. In the end of this note, we explain how to use the chord index $f_{\gamma}$ to define a more general transcendental function virtual knot invariant. The main idea is to modify the definition of Ind$(c)$ a little bit, instead of counting the number of intersections between the chord $c$ and other chords, we count the weighted sum of these intersections where each weight is derived from $f_{\gamma}$.

As before, let $K$ be a virtual knot which possesses a knot diagram (also denoted by $K$) on a minimal surface $\Sigma$. Choose a closed curve $\gamma\subset\Sigma$ which satisfies $[\gamma]\in\mathcal{H}_1^K(\Sigma, \mathbb{Z})$. Now for each crossing point $c$, we obtain a chord index $f_{\gamma}(c)$ and use $\phi$ to denote the quotient map from $\mathbb{Z}$ to $\mathbb{Z}_{|f_{\gamma}(c)|}$. For instance, if $f_{\gamma}(c)=0$ then $\phi=id$. If $f_{\gamma}(c)=\pm1$, then $\phi$ sends every integer to zero.

Consider the Gauss diagram $G(K)$ and fix a chord $c$, let us use $\{r_1, \cdots, r_m\}$ and $\{l_1, \cdots, l_n\}$ to denote the set of chords crossing $c$ from left to right and the set of chords crossing $c$ from right to left, respectively. Obviously, it follows that
\begin{center}
$r_+(c)+r_-(c)=m$, $l_+(c)+l_-(c)=n$ and $\sum\limits_{i=1}^mw(r_i)-\sum\limits_{j=1}^nw(l_j)=\text{Ind}(c)$.
\end{center}
Now we define an \emph{index function with respect to $\gamma$} as follows
\begin{center}
$g_c^{\gamma}(s)=\sum\limits_{i=1}^mw(r_i)s^{\phi(f_{\gamma}(r_i))}-\sum\limits_{j=1}^nw(l_j)s^{\phi(-f_{\gamma}(l_j))}\in\mathbb{Z}[s^{\pm1}]/(s^{f_{\gamma}(c)}-1)$.
\end{center}
It is evident that Ind$(c)$ can be recovered from $g_c^{\gamma}(s)$ by setting $s=1$. By using $g_c^{\gamma}(s)$, we introduce the following transcendental function
\begin{center}
$F_K^{\gamma}(\textbf{t}, s)=\sum\limits_{k\in\mathbb{Z}}\sum\limits_{f_{\gamma}(c)=k}w(c)t_k^{g_c^{\gamma}(s)}-w(K)t_0^0$,
\end{center}
where the second sum runs over all crossing points $c$ satisfying $f_{\gamma}(c)=k$ and $\textbf{t}$ denotes $\{t_i\}_{i\in\mathbb{Z}}$.

It is worthy to remark that, strictly speaking, $F_K^{\gamma}(\textbf{t}, s)$ is just a formal sum rather than a transcendental function in the traditional sense. Since for different $i$ and $ j$, $t_i$ and $t_j$ live in different rings. In particular, the equality $t_i^0=1$ does not hold here, since there is no so called constant term in $F_K^{\gamma}(\textbf{t}, s)$. For example, $t_i^0-t_j^0\neq0$ if $i\neq j$.

\begin{proposition}\label{proposition5.6}
The index function $g_c^{\gamma}(s)$ satisfies all the chord axioms listed in Section \ref{section4}. As a corollary, $F_K^{\gamma}(\textbf{t}, s)$ is a virtual knot invariant.
\end{proposition}
\begin{proof}
The proof is analogous to that of Theorem 3.1 in \cite{Che2017}. We sketch it here.
\begin{itemize}
  \item For Reidemeister move of type I, the crossing point $c$ involved has chord index $f_{\gamma}(c)=0$. On the other hand, besides of $c$, no chord in $G(K)$ has nonempty intersection with $c$, therefore $g_c^{\gamma}(s)=0$. It follows that the contribution coming from $c$ to $F_K^{\gamma}(\textbf{t}, s)$ equals $w(c)t_0^0-w(c)t_0^0=0$.
  \item For Reidemeister move of type II, it suffices to notice that the two crossing points $c_1$ and $c_2$ have the same chord indices, i.e. $f_{\gamma}(c_1)=f_{\gamma}(c_2)$ but $w(c_1)=-w(c_2)$. On the other hand, any other chord either has nonempty intersection with both $c_1$ and $c_2$, or with none of them. We conclude that $g_{c_1}^{\gamma}(s)=g_{c_2}^{\gamma}(s)$, and the contributions coming from $c_1$ and $c_2$ to $F_K^{\gamma}(\textbf{t}, s)$ cancel out.
  \item For Reidemeister move of type III, we only check the left diagram illustrated in Figure \ref{figure5}. The other case can be verified similarly. Figure \ref{figure10} depicts two Gauss diagrams related by one  Reidemeister move of type III. As we mentioned in the proof of Lemma \ref{lemma2.2}, the chord index $f_{\gamma}(c_i)$ $(1\leq i\leq3)$ are preserved respectively under $\Omega_3$.
  \begin{figure}[h]
  \centering
  \includegraphics{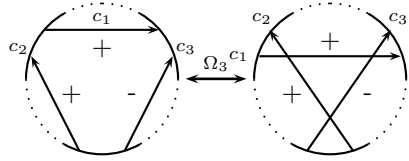}\\
  \caption{Two Gauss diagrams related by $\Omega_3$}\label{figure10}
  \end{figure}

  For the Gauss diagram on the left side of Figure \ref{figure10}, we assume that
  \begin{center}
  $g_{c_1}^{\gamma}(s)=a_1(s), g_{c_2}^{\gamma}(s)=a_2(s)$ and $g_{c_3}^{\gamma}(s)=a_3(s)$.
  \end{center}
  Then for the Gauss diagram on the right side of Figure \ref{figure10}, we have
  \begin{center}
  $g_{c_1}^{\gamma}(s)=a_1(s)+s^{\phi_{c_1}(f_{\gamma}(c_2))}-s^{\phi_{c_1}(f_{\gamma}(c_3))}$,

  $g_{c_2}^{\gamma}(s)=a_2(s)-s^{\phi_{c_2}(-f_{\gamma}(c_1))}+s^{\phi_{c_2}(-f_{\gamma}(c_3))}$,

  $g_{c_3}^{\gamma}(s)=a_3(s)-s^{\phi_{c_3}(-f_{\gamma}(c_1))}+s^{\phi_{c_3}(f_{\gamma}(c_2))}$,
  \end{center}
  where $\phi_{c_i}$ $(1\leq i\leq3)$ denotes the quotient map from $\mathbb{Z}$ to $\mathbb{Z}_{|f_{\gamma}(c_i)|}$. Recall that $f_{\gamma}(c_3)=f_{\gamma}(c_1)+f_{\gamma}(c_2)$, then for the Gauss diagram on the right side we have
  \begin{center}
  $g_{c_1}^{\gamma}(s)=a_1(s)+s^{\phi_{c_1}(f_{\gamma}(c_2))}-s^{\phi_{c_1}(f_{\gamma}(c_1)+f_{\gamma}(c_2))}=a_1(s)$,

  $g_{c_2}^{\gamma}(s)=a_2(s)-s^{\phi_{c_2}(-f_{\gamma}(c_1))}+s^{\phi_{c_2}(-f_{\gamma}(c_1)-f_{\gamma}(c_2))}=a_2(s)$,

  $g_{c_3}^{\gamma}(s)=a_3(s)-s^{\phi_{c_3}(f_{\gamma}(c_2))}+s^{\phi_{c_3}(f_{\gamma}(c_2))}=a_3(s)$.
  \end{center}
  We conclude that each $g_{c_i}^{\gamma}(s)$ $(1\leq i\leq3)$ is preserved under $\Omega_3$, which follows that the contributions coming from $c_1, c_2$ and $c_3$ to $F_K^{\gamma}(\textbf{t}, s)$ are preserved respectively. On the other hand, the contribution from all other chords not involved in $\Omega_3$ is invariant under $\Omega_3$. The proof is complete.
\qedhere
\end{itemize}
\end{proof}

\begin{remark}
If one chooses a closed curve representing $[K]$, then $F_K^{\gamma}(\textbf{t}, s)$ reduces to the transcendental function invariant introduced in \cite{Che2017}. We remark that $F_K^{\gamma}(\textbf{t}, s)$ also can be regarded as a knot invariant of knots in $\Sigma\times I$.
\end{remark}

\section*{Acknowledgement}
Zhiyun Cheng and Hongzhu Gao are supported by NSFC 11771042 and NSFC 12071034. Mengjian Xu is supported by NSFC 12001124.

\end{document}